\documentclass{article}
\usepackage{amsmath,amssymb}

\newtheorem{theorem}{Theorem}[section]
\newtheorem{lemma}[theorem]{Lemma}
\newtheorem{proposition}[theorem]{Proposition}
\newtheorem{corollary}[theorem]{Corollary}

\newcommand{\lin}{\operatorname{lin}}

\newcommand{\gr}{\operatorname{gr}}

\newcommand{\Sym}{\operatorname{Sym}}

\newenvironment{proof}{\par\noindent{\bf Proof.}}{$\qed$\par\bigskip}
\newcommand{\qed}{\enspace\vrule  height6pt  width4pt  depth2pt}
\begin{document}
\setcounter{page}{0} \thispagestyle{empty}

\title{Finitely presented algebras and groups defined by permutation relations \thanks{
Research  partially supported by grants of
MICIN-FEDER (Spain) MTM2008-06201-C02-01,
Generalitat de Catalunya 2005SGR00206,
Onderzoeksraad of Vrije Universiteit Brussel, Fonds
voor Wetenschappelijk Onderzoek (Belgium),
Flemish-Polish bilateral agreement BIL2005/VUB/06
and MNiSW research grant N201 004 32/0088
(Poland)}.}

\author{Ferran Ced\'o \and Eric Jespers \and Jan Okni\'nski}
\date{}
\maketitle

\begin{abstract}
The class of finitely presented algebras over a
field $K$ with a set of generators $a_{1},\ldots ,
a_{n}$ and defined by homogeneous relations of the
form
 $a_{1}a_{2}\cdots a_{n} =a_{\sigma (a)} a_{\sigma (2)} \cdots
 a_{\sigma (n)}$,
where $\sigma$ runs through a subset $H$ of the
symmetric group $\Sym_{n}$ of degree $n$, is
introduced. The emphasis is on the case of a cyclic
subgroup $H$ of $\Sym_{n}$ of order $n$. A normal
form of elements of the algebra is obtained. It is
shown that the underlying monoid, defined by the
same (monoid) presentation, has a group of fractions
and this group is described. Properties of the
algebra are derived. In particular, it follows that
the algebra is a semiprimitive domain. Problems
concerning the groups and algebras defined by
arbitrary subgroups $H$ of $\Sym_{n}$ are proposed.
\end{abstract}

{\it keywords: symmetric presentation, finitely
presented, semigroup algebra, monoid, group}

\section{Introduction}
Recently there has been extensive interest in some
finitely presented algebras $A$ defined by
homogeneous semigroup relations, that is, relations
of the form $w=v$, where $w$ and $v$ are words of
the same length in a generating set of the algebra.
A particular intriguing class (defined by
homogeneous relations of degree $2$) are the
algebras yielding set theoretic solutions of the
Yang-Baxter equation,
\cite{eting,gat-sol,gat-jes-okn,gat-van,jes-okn-Itype,
bookspringer,rump}. Such algebras do not only have
nice ring theoretic properties but they also lead to
exciting groups $G$ and monoids $S$ defined by the
same presentation. Other classes of algebras in this
context and defined by homogeneous relations of
degree $3$ are Chinese \cite{chinese,jaszu-okn} and
plactic algebras \cite{plactic,las-lec}. The origin
for studying these comes from aspects in Young
diagrams, representation theory and algebraic
combinatorics, see for example \cite{fulton}. In all
the mentioned algebras there is a strong connection
between the structure of the algebra $A$ and the
underlying semigroup $S$ and group $G$. The study of
$S$ and $G$ has often led to a successful approach
in the investigations of the algebraic structure of
$A$. For example, the structure of primes of $S$ and
primes of $A$ is related and the minimal primes of
$A$ have an unexpectedly nice structure. Also,
information on the prime and the Jacobson radicals
of $A$ can be deduced.

The aim of this paper is to introduce and
investigate a new class of finitely presented
algebras over a field $K$ with a set of generators
$a_{1},\ldots , a_{n}$ satisfying homogeneous
relations of the form
 \begin{eqnarray} a_{1}a_{2}\cdots a_{n} &=&a_{\sigma (a)} a_{\sigma (2)} \cdots
 a_{\sigma (n)}, \label{present}
 \end{eqnarray}
where $\sigma$ runs through a subset $H$ of the
symmetric group $\Sym_{n}$ of degree $n$. Note that
such an algebra  is  the semigroup algebra
$K[S_{n}(H)]$, where $S_{n}(H)$ is the monoid
defined by the same presentation. By $G_{n}(H)$ we
denote the group defined by this presentation. We
initiate a study of combinatorial and algebraic
aspects of the algebras $K[S_{n}(H)]$ and of the
associated monoids $S_{n}(H)$ and groups $G_{n}(H)$.
Note that various other definitions and examples of
symmetric presentations of groups have been
considered in the literature, see for example
\cite{curtis} and \cite{johnson-odoni}.

We mention two easy examples: $K[S_{n}(\{
1\})]=K\langle a_{1},\ldots , a_{n}\rangle$, the
free $K$-algebra, $S_{n}(\{ 1 \})=FM_{n}$ the free
monoid of rank $n$ and
$K[S_{2}(\Sym_{2})]=K[a_{1},a_{2}]$, the commutative
polynomial algebra in two variables. In
Proposition~\ref{symmetriccase} we prove that the
latter can be extended as follows:
$K[S_{n}(\Sym_{n})]$ is a subdirect product of a
commutative polynomial algebra in $n$ variables and
a primitive monomial algebra. A proof for this
relies on a recent result of Bell and Pekcagliyan
\cite{bell}, see also \cite{okn_monomial}. Such
monomial algebras are of independent interest and
they are studied, for example, by Okni\'{n}ski
\cite{book}, Bell and Smoktunowicz \cite{bellsmokt}
and  Belov, Borisenko and Latyshev \cite{belov}. For
other interesting classes $K[S]$ of algebras defined
by homogeneous semigroup presentations, the study of
minimal prime ideals naturally leads to the ideals
intersecting $S$ and the ideals of $K[S/\rho]$,
where $\rho$ is the least cancellative congruence on
$S$, and therefore in many cases to ideals of the
group algebra $K[G]$, where $G$ is the group of
fractions of $S/\rho$. This approach has proved
quite successful in many classes of such algebras
and the results in this paper indicate that this
might also be the case for general $K[S_{n}(H)]$,
that is, with $H$ a nontrivial proper subset of
$\Sym_{n}$. The most promising case seems to be when
$H$ is a subgroup of $\Sym_{n}$.

We will give a detailed account in case $H$ is the
cyclic group generated by the cycle $(1,2,\ldots
,n)\in \Sym_{n}$. We describe the associated group
$G_{n}(H)$ and the monoid $S_{n}(H)$. In particular,
we obtain a normal form for the elements of the
algebra $K[S_{n}(H)]$. Further, we derive some
structural properties of this algebra. In
particular, we show that it is a semiprimitive
domain. We conclude with a list of problems
concerning the general case of groups $G_{n}(H)$ and
algebras $K[S_{n}(H)]$, with $H$ an arbitrary
subgroup of $\Sym_{n}$. These are on one hand
motivated by the results of this paper and on the
other hand by the nice behavior of the prime ideals
and of the radical in some other interesting classes
of finitely presented algebras defined by
homogeneous relations,
\cite{gat-jes-okn,jaszu-okn,jes-okn-Itype}. The
latter share the flavor of semigroup algebras $K[S]$
of semigroups $S$ that satisfy a permutational
identity (a much stronger requirement than the one
considered in this paper). For results on these we
refer for example to \cite{nordahl,book,putcha}. Our
results might indicate that the considered algebras
often lead to the following two constructions, that
are of independent interest: the algebras $K[M/\rho
]$, where $\rho $ is the least cancellative
congruence on $M$ and the monomial algebras.

Throughout the paper $K$ is a field.  If
$b_{1},\ldots, b_{m}$ are elements of a monoid $M$
then we denote by $\langle b_{1},\ldots ,
b_{m}\rangle$ the submonoid generated by
$b_{1},\ldots , b_{m}$. If $M$ is a group then $\gr
(b_{1},\ldots ,b_{m})$ denotes the subgroup of $M$
generated by $b_{1},\ldots, b_{m}$.

Clearly, the defining relations of an arbitrary
$S_{n}(H)$ are homogeneous. Hence, it has a natural
degree or length function. This will be used freely
throughout the paper.

\section{The case of a cyclic group of order $n$
}\label{example}

As mentioned in the introduction, if $n=2$ then
$S_{2}(H)$ is the free monoid if $H=\{ 1\}$ and
otherwise $S_{2}(H)$ is the free abelian monoid of
rank $2$. So, from now on, we assume that $n\geq 3$.
In this section we study $S_{n}(H)$ in case $H$ is
the cyclic subgroup of $\Sym_{n}$ generated by the
cycle $(1\, , 2\, , \ldots ,n)$. For simplicity, we
denote $S_{n}(\gr (\{ (1\, ,2\, , \ldots ,n)\} ))$
by $S_{n}$.

First, we obtain a normal form for the elements of
$S_{n}$. This allows us to show next that $S_{n}$ is
a submonoid of a group, actually it has a group of
fractions which is obtained by localizing at a
central infinite cyclic submonoid.

So,  throughout the paper, $S_{n}$ is the monoid
defined by generators $a_1,a_2,\dots ,a_n$ and
relations $$a_1a_2\cdots a_n=a_{\sigma (1)}a_{\sigma
(2)}\cdots a_{\sigma (n)},$$ where
$\sigma=(1,2,\dots ,n)^{i}\in \Sym_n$ and
$i=1,2,\dots ,n$.

 We first investigate the structure of the monoid
$S_{n}$. We begin by noting that $a_1a_2\cdots a_n$
is a central element of $S_n$. Indeed, for all
$1\leq i\leq n-1$, $$(a_1a_2\cdots
a_n)a_i=a_ia_{i+1}\cdots a_na_1\cdots
a_i=a_i(a_1a_2\cdots a_n),$$ and $a_na_1a_2\cdots
a_{n-1}a_n=(a_1a_2\cdots a_n)a_n$.

\begin{theorem}\label{normalform}
Each element $a\in S_n$ can be uniquely written as a
product
\begin{eqnarray}\label{nf}
a=a_1^{i}(a_1a_2\cdots a_n)^{\varepsilon}(a_2\cdots
a_n)^jb,
\end{eqnarray}
where  $\varepsilon\in\{ 0,1\}$, $b\in S_n\setminus
(a_1S_n\cup a_2\cdots a_nS_n)$ and $i,j$ are
non-negative integers so that the following property
holds  $$j\geq 1\Rightarrow \varepsilon =1\;\mbox{
or }\; i=0.$$
\end{theorem}

\begin{proof}
Consider the following transformations of words in
the free monoid $FM_n$ on $\{ a_1,a_2,\dots ,a_n\}$
(we use the same notation for the generators of
$FM_{n}$ and those of $S_{n}$):
\begin{eqnarray}\label{ti}
t_i(a_{i+1}\cdots a_na_1a_2\cdots a_i)=a_1a_2\cdots
a_n,
\end{eqnarray}
for $i=1,2,\dots ,n-1$,
\begin{eqnarray}\label{rjm}
r_{j,m}(a_{j}a_1^{m}a_2\cdots a_n)=a_1a_2\cdots
a_na_ja_1^{m-1},
\end{eqnarray}
for $j=2,\dots ,n-1$ and $m\geq 1$, and
\begin{eqnarray}\label{rmn}
r_{n,m}(a_{n}a_1^{m}a_2\cdots a_n)=a_1a_2\cdots
a_na_na_1^{m-1},
\end{eqnarray}
for $m\geq 2$.

By the comment before the theorem, each of these
transformations maps a word in $FM_{n}$ to another
word that represents the same element in $S_{n}$.

Let $w_1,w_2\in FM_n$. We say that $w_1$ covers
$w_2$ if $w_2$ is obtained from $w_1$ by applying
exactly one of the transformations $t_i,r_{k,m}$ to
a subword of $w_1$. In this case we write $w_1\succ
w_2$. We define a preorder in $FM_n$ by $w\geq w'$
if and only if there exist $w_0,w_1,\dots ,w_r\in
FM_n$, ($r\geq 0$), such that $$w=w_0\succ
w_1\succ\dots\succ w_r=w'.$$ We will see that this
preorder satisfies conditions (i) and (ii) of the
Diamond lemma \cite[Theorem~I.4.9]{Cohn}:
\begin{itemize}
\item[{\rm (i)}] Denote by $\ll$ the length-lexicographical order
on $FM_n$ generated by $$a_1\ll a_2\ll\dots\ll
a_n.$$ Note that for $w,w'\in FM_n$, $$w\geq
w'\Rightarrow w\gg w'.$$ Hence, for each $w\in
FM_n$, there exists a positive integer $k(w)$ such
that every descending chain starting with $w$,
$$w=w_0\geq w_1\geq\dots$$ where $w_{i-1}\neq w_i$,
has at most $k(w)$ terms. (In fact $k(w)\leq m!$,
where $m$ is the length of $w$ because the
transformations (\ref{ti}),(\ref{rjm}),(\ref{rmn})
are permutations.)

\item[{\rm (ii)}] Let $w,w_1,w_2\in FM_n$ such that $w\succ w_1$
and $w\succ w_2$. We should show that there exists
$w_{3}\in FM_{n}$ such that $w_{1},w_{2}\geq w_{3}$.
We know that $w=a_{i_1}a_{i_2}\cdots a_{i_m}$,
$$w_1=a_{i_1}\cdots a_{i_s}f(a_{i_{s+1}}\cdots
a_{i_{s+p}})a_{i_{s+p+1}}\cdots a_{i_m}$$ and
$$w_2=a_{i_1}\cdots a_{i_u}g(a_{i_{u+1}}\cdots
a_{i_{u+q}})a_{i_{u+q+1}}\cdots a_{i_m},$$ where
$a_{i_1},\dots ,a_{i_m}\in \{a_1,\dots ,a_n\}$ and
$f,g\in \{t_i\mid 1\leq i\leq n-1\}\cup\{
r_{j,m}\mid 2\leq j\leq n-1$ and $m\geq 1\}\cup \{
r_{n,m}\mid m\geq 2\}$.

Note that if $s+p<u+1$ or $u+q<s+1$, then applying
$g$ to the subword $a_{i_{u+1}}\cdots a_{i_{u+q}}$
of $w_1$ and $f$ to the subword $a_{i_{s+1}}\cdots
a_{i_{s+p}}$ of $w_2$ we obtain the same word $w_3$
and $$w_1\succ w_3\quad\mbox{and}\quad w_2\succ
w_3.$$

If $s+p\geq u+1$ and $u+q\geq s+1$, we say that the
subwords $a_{i_{s+1}}\cdots a_{i_{s+p}}$ and
$a_{i_{u+1}}\cdots a_{i_{u+q}}$ of $w$ overlap. We
will  study all possible overlaps between subwords
of the forms
\begin{itemize}
\item[$(\alpha)$] $a_{i+1}\cdots a_na_1\cdots a_i$ for $1\leq
i\leq n-1$,

\item[$(\beta)$] $a_{j}a_1^{m}a_2\cdots a_n$ for $2\leq j\leq n-1$
and $m\geq 1$,

\item[$(\gamma)$] $a_{n}a_1^{m}a_2\cdots a_n$ for $m\geq 2$.
\end{itemize}
Note that overlaps between two different subwords of
the form $(\beta)$ cannot occur.

{\it Case 1:} Overlaps between two subwords of the
form $(\alpha)$.

In this case we may assume that $1\leq i<k\leq n-1$,
$$w=a_{i+1}\cdots a_{k+1}\cdots a_na_1\cdots
a_i\cdots a_k,$$ $$w_1=a_{1}a_2\cdots
a_na_{i+1}\cdots a_k\quad\mbox{and}\quad
w_2=a_{i+1}\cdots a_{k}a_1a_2\cdots a_n.$$ Now we
have (applying $r_{k,1},r_{k-1,1},\dots ,r_{i+1,1}$)
\begin{eqnarray*}
w_2&=&a_{i+1}\cdots a_{k}a_1a_2\cdots a_n\succ
a_{i+1}\cdots a_{k-1}a_1a_2\cdots a_na_k\\ &\succ
&\dots\succ a_{i+1}a_1a_2\cdots a_na_{i+2}\cdots
a_{k}\\ &\succ & a_1a_2\cdots a_na_{i+1}\cdots
a_{k}=w_1.
\end{eqnarray*}

{\it Case 2:} Overlaps between a subword of the form
$(\alpha)$ and a subword of the form $(\beta)$.

Assume first that $2\leq j\leq n-1$, $m\geq 1$,
$$w=a_ja_1^ma_2\cdots a_{n}a_1,$$
$$w_1=a_{1}a_2\cdots
a_na_{j}a_1^{m-1}a_1\quad\mbox{and}\quad
w_2=a_ja_1^ma_1a_2\cdots a_n.$$ In this case
$w_2\succ w_1$ (applying $r_{j,m+1}$).

Assume now that $2\leq j\leq n-1$, $2\leq i\leq
n-1$, $m\geq 1$, $$w=a_ja_1^ma_2\cdots
a_{n}a_1\cdots a_i,$$ $$w_1=a_{1}a_2\cdots
a_na_{j}a_1^{m-1}a_1\cdots a_i\quad\mbox{and}\quad
w_2=a_ja_1^ma_2\cdots a_{i}a_1a_2\cdots a_n.$$ In
this case we have
\begin{eqnarray*}
w_2&=&a_ja_1^ma_2\cdots a_{i}a_1a_2\cdots a_n\succ
a_ja_1^ma_2\cdots a_{i-1}a_1a_2\cdots a_na_i\\
&\succ&\dots\succ a_ja_1^ma_1a_2\cdots a_na_2\cdots
a_i\\ &\succ& a_1a_2\cdots a_na_ja_1^{m}a_2\cdots
a_i=w_1
\end{eqnarray*}
(applying $r_{i,1},r_{i-1,1},\dots
,r_{2,1},r_{j,m+1}$).

Assume now that $2\leq j\leq n-1$, $m\geq 1$,
$$w=a_{j+1}\cdots a_na_1a_2\cdots
a_{j}a_1^{m}a_2\cdots a_n,$$ $$w_1=a_{1}a_2\cdots
a_na_1^{m}a_2\cdots a_n$$ and $$w_2=a_{j+1}\cdots
a_{n}a_1\cdots a_{j-1}a_1a_2\cdots
a_na_ja_1^{m-1}.$$ In this case we have (applying
$t_{n-1}$, if $m=1$, or $r_{n,m}$, if $m\geq 2$)
\begin{eqnarray*}
w_1=a_{1}a_2\cdots a_na_1^{m}a_2\cdots a_n\succ
a_{1}a_2\cdots a_{n-1}a_1a_2\cdots a_na_na_1^{m-1},
\end{eqnarray*}
and (applying $r_{j-1,1},\dots
,r_{2,1},r_{n,2},r_{n-1,1},\dots ,r_{j+1,1}, t_{j},
t_{n-1}$)
\begin{eqnarray*}
w_2&=&a_{j+1}\cdots a_{n}a_1\cdots
a_{j-1}a_1a_2\cdots a_na_ja_1^{m-1}\\
&\succ&\dots\succ a_{j+1}\cdots a_na_1^2a_2\cdots
a_na_2\cdots a_ja_1^{m-1}\\ &\succ& a_{j+1}\cdots
a_{n-1}a_1a_2\cdots a_na_na_1a_2\cdots
a_ja_1^{m-1}\\ &\succ&\dots\succ a_1a_2\cdots
a_na_{j+1}\cdots a_na_1\cdots a_ja_1^{m-1}\\ &\succ&
a_1a_2\cdots a_na_1a_2\cdots a_na_1^{m-1}\\ &\succ&
a_1a_2\cdots a_{n-1}a_1a_2\cdots a_na_na_1^{m-1}.
\end{eqnarray*}
Hence, in this case,
 $$w_1\geq a_1a_2\cdots
a_{n-1}a_1a_2\cdots a_na_na_1^{m-1}$$ and $$ w_2\geq
a_1a_2\cdots a_{n-1}a_1a_2\cdots a_na_na_1^{m-1}.$$

{\it Case 3:} Overlaps between a subword of the form
$(\alpha)$ and a subword of the form $(\gamma)$.

Assume first that $m\geq 2$, $$w=a_na_1^ma_2\cdots
a_{n}a_1,$$ $$w_1=a_{1}a_2\cdots
a_na_{n}a_1^{m-1}a_1\quad\mbox{and}\quad
w_2=a_na_1^ma_1a_2\cdots a_n.$$ In this case
$w_2\succ w_1$ (applying $r_{n,m+1}$).

Assume now that $2\leq i\leq n-1$, $m\geq 2$,
$$w=a_na_1^ma_2\cdots a_{n}a_1\cdots a_i,$$
$$w_1=a_{1}a_2\cdots a_na_{n}a_1^{m-1}a_1\cdots
a_i\quad\mbox{and}\quad w_2=a_na_1^ma_2\cdots
a_{i}a_1a_2\cdots a_n.$$ In this case we have
(applying $r_{i,1},r_{i-1,1},\dots
,r_{2,1},r_{n,m+1}$)
\begin{eqnarray*}
w_2&=&a_na_1^ma_2\cdots a_{i}a_1a_2\cdots a_n\succ
a_na_1^ma_2\cdots a_{i-1}a_1a_2\cdots a_na_i\\
&\succ&\dots\succ a_na_1^ma_1a_2\cdots a_na_2\cdots
a_i\\ &\succ& a_1a_2\cdots a_na_na_1^{m}a_2\cdots
a_i=w_1.
\end{eqnarray*}

{\it Case 4:} Overlaps between a subword of the form
$(\beta)$ and a subword of the form $(\gamma)$.

In this case we may assume that $1\leq j\leq n-1$,
$m\geq 1$, $p\geq2$, $$w=a_{j}a_1^{m}a_2\cdots
a_{n}a_1^{p}a_2\cdots a_n,$$ $$w_1=a_{1}a_2\cdots
a_na_{j}a_1^{m-1}a_1^{p}a_2\cdots a_n$$ and
$$w_2=a_{j}a_1^{m}a_2\cdots a_{n-1}a_1a_2\cdots
a_na_na_1^{p-1}.$$ Now we have (applying
$r_{j,m+p-1}$)
\begin{eqnarray*}
w_1=a_{1}a_2\cdots a_na_{j}a_1^{m-1}a_1^{p}a_2\cdots
a_n\succ a_{1}a_2\cdots a_na_1a_2\cdots
a_na_{j}a_1^{m+p-2},
\end{eqnarray*}
and (applying $r_{n-1,1},\dots
,r_{2,1},r_{j,m+1},r_{j,m}$)
\begin{eqnarray*}
w_2&=&a_{j}a_1^{m}a_2\cdots a_{n-1}a_1a_2\cdots
a_na_na_1^{p-1}\\ &\succ &a_{j}a_1^{m}a_2\cdots
a_{n-2}a_1a_2\cdots a_na_{n-1}a_na_1^{p-1}\\ &\succ
&\dots\succ a_{j}a_1^{m}a_1a_2\cdots a_na_2\cdots
a_na_1^{p-1}\\ &\succ &a_1a_2\cdots
a_na_{j}a_1^{m}a_2\cdots a_na_1^{p-1}\\ &\succ
&a_1a_2\cdots a_{n}a_1a_2\cdots a_na_{j}a_1^{m+p-2}.
\end{eqnarray*}
Hence $$w_1\geq a_{1}a_2\cdots a_na_1a_2\cdots
a_na_{j}a_1^{m+p-2}$$ and $$w_2\geq a_{1}a_2\cdots
a_na_1a_2\cdots a_na_{j}a_1^{m+p-2}$$ in this case.

{\it Case 5:} Overlaps between two subwords of the
form $(\gamma)$.

In this case we may assume that  $m,p\geq 2$,
$$w=a_{n}a_1^{m}a_2\cdots a_{n}a_1^{p}a_2\cdots
a_n,$$ $$w_1=a_{1}a_2\cdots
a_na_{n}a_1^{m-1}a_1^{p}a_2\cdots a_n$$ and
$$w_2=a_{n}a_1^{m}a_2\cdots a_{n-1}a_1a_2\cdots
a_na_na_1^{p-1}.$$ Similarly as in Case 4, we obtain
that $$w_1\geq a_{1}a_2\cdots a_na_1a_2\cdots
a_na_{n}a_1^{m+p-2}$$ and $$w_2\geq a_{1}a_2\cdots
a_na_1a_2\cdots a_na_{n}a_1^{m+p-2}$$ in this case.
\end{itemize}
Therefore, by the Diamond lemma
\cite[Theorem~I.4.9]{Cohn}, if $\pi \colon
FM_n\longrightarrow S_n$ is the natural projection,
for each $a\in S_n$ there exists a unique $w\in
FM_n$ such that $\pi (w)=a$ and $w$ is a minimal
element in $(FM_n, \leq)$. It is easy to see that
the minimal elements in $(FM_n, \leq)$ are of the
form
\begin{eqnarray}  \label{form}
a_1^{i}(a_1a_2\cdots a_n)^{\varepsilon}(a_2\cdots
a_n)^jw,
\end{eqnarray}
where $i,j$ are non-negative integers,
$\varepsilon\in\{ 0,1\}$ and $w\in FM_n\setminus
(I\cup a_1FM_n\cup a_2\cdots a_nFM_n)$, where $I$ is
the ideal of $FM_n$ generated by $a_{\sigma
(1)}\cdots a_{\sigma (n)}$ for all
$\sigma=(1,2,\dots ,n)^{k}\in \Sym_n$ and
$k=1,2,\dots ,n$. Moreover, if we also assume that
$$j\geq 1\Rightarrow \varepsilon =1\;\mbox{ or }\;
i=0$$ then (\ref{form}) gives a unique presentation
of all minimal elements in $(FM_n, \leq)$. Since
$\pi(I)=a_1a_2\cdots a_nS_n$, the result follows.
\end{proof}

Note that, using the notation in the last part of
the proof of the above theorem, $S_n\setminus
a_1a_2\cdots a_nS_n=\{ a\in S_n\mid
|\pi^{-1}(a)|=1\}$. So, elements in this set have a
unique presentation in $S_{n}$.  In particular, we
have a natural isomorphism $S_{n}/(a_{1}\cdots
a_{n}S_{n})\cong FM_{n}/\pi^{-1} (a_{1}\cdots
a_{n}S_{n})$ (again, we use the same notation for
the generators of $FM_{n}$ and those of $S_{n}$).

\begin{theorem} \label{cancel}
The monoid $S_{n}$ is cancellative and it has a
group $G$ of fractions of the form $G=S_{n}\langle
a_{1}\cdots a_{n}\rangle ^{-1}\cong F\times C$,
where $F=\gr (a_{1},\ldots , a_{n-1})$ is a free
group of rank $n-1$ and $C=\gr (a_{1}\cdots a_{n})$
is a cyclic infinite group.
\end{theorem}

\begin{proof}
Let $w_{k}=a_1^{i_{k}}(a_1a_2\cdots
a_n)^{\varepsilon_{k}}(a_2\cdots a_n)^{j_{k}}b_{k},
k=1,2$, be two elements of $S_{n}$ written in the
canonical form established in
Theorem~\ref{normalform}. Suppose that
$w_{1}a_{m}=w_{2}a_{m}$ for some $m\in \{1,\ldots
,n\}$. Then $w_{1}(a_{1}\cdots
a_{n})=w_{2}(a_{1}\cdots a_{n})$. On the other hand,
we have $$w_{k}(a_{1}\cdots
a_{n})=a_{1}^{i_{k}}(a_{1}\cdots a_{n})(a_{2}\cdots
a_{n})^{j_{k}}b_k$$ if $\varepsilon_{k}=0$ (and then
$i_{k}=0$ or $j_{k}=0$) and $$w_{k}(a_{1}\cdots
a_{n})=a_{1}^{i_{k}+1}(a_{1}\cdots
a_{n})(a_{2}\cdots a_{n})^{j_{k}+1}b_k$$ if
$\varepsilon_{k}=1$. Therefore, if $i_{1}=0$ or
$j_{1}=0$ then $i_{2}=0$ or $j_{2}=0$ and
$$a_{1}^{i_{1}}(a_{1}\cdots a_{n})(a_{2}\cdots
a_{n})^{j_{1}}b_1=a_{1}^{i_{2}}(a_{1}\cdots
a_{n})(a_{2}\cdots a_{n})^{j_{2}}b_2,$$ which
implies that $w_{1}=w_{2}$. If $i_{1}$ and $j_{1}$
are nonzero, then also $i_{2}$ and $j_{2}$ are
nonzero and we get $$a_{1}^{i_{1}+1}(a_{1}\cdots
a_{n})(a_{2}\cdots
a_{n})^{j_{1}+1}b_1=a_{1}^{i_{2}+1}(a_{1}\cdots
a_{n})(a_{2}\cdots a_{n})^{j_{2}+1}b_2.$$ So, again
$w_{1}=w_{2}$. It follows that $S_{n}$ is right
cancellative.

It is clear that for every $w\in S_{n}$ there exists
$u\in S_{n}$ such that $wu=(a_{1}\cdots a_{n})^{k}$
for some nonnegative integer $k$. Since
$(a_{1}\cdots a_{n})^{k}$ is central, it follows
that  $S_{n}$ is cancellative and that the central
localization $S_{n}\langle a_{1}\cdots a_{n}\rangle
^{-1}$ is the group of fractions of $S_{n}$.

Let $F$ be the group generated by elements
$x_{1},\ldots ,x_{n}$ subject to the relation
$x_{1}\cdots x_{n}=1$. Clearly, $F$ is a free group
of rank $n-1$. Let $C$ be an infinite cyclic group
with a generator $c$. Since $x_{i}\cdots
x_{n}cx_{1}\cdots x_{i-1}=x_{i}\cdots
x_{n}x_{1}\cdots x_{i-1}c= c$ in $F\times C$ for
every $i\in \{2,\ldots ,n\}$, it is clear that the
rules $\phi (a_{i})=x_{i}, i=1,\ldots ,n-1$, $\phi
(a_{n})=x_{n}c$ determine a homomorphism $\phi
:S_{n}\longrightarrow F\times C$. It extends to a
homomorphism $\overline{\phi} :G\longrightarrow
F\times C$. On the other hand, the rules: $\psi
(x_{i})= a_{i}, i=1,\ldots, n-1$, and $\psi
(c)=a_{1}\cdots a_{n}$, determine a homomorphism
$\psi : F\times C\longrightarrow G$. It is easy to
see that this is the inverse of $\overline{\phi}$.
The assertion follows.
\end{proof}

Next we investigate the structure of the semigroup
algebra $K[S_n]$. We start with an easy consequence
of Theorem~\ref{cancel}.

\begin{corollary}
The algebra $K[S_{n}]$ is a domain and it is
semiprimitive.
\end{corollary}

\begin{proof}
By Theorem~\ref{cancel}, the group of  fractions of
$S_{n}$ is a unique product group. Hence, $S_{n}$ is
a unique product monoid. Therefore the assertions
follow from Theorem~10.4, Corollary~10.5 and
Theorem~10.6 in \cite{book}.
\end{proof}

We know from Theorem~\ref{cancel} that $S_{n}$ has a
group of fractions $G$ that is a central
localization. Hence, for a field $K$, the group
algebra $K[G]$ is a central localization of
$K[S_{n}]$. It follows that if $K[S_{n}]$ is a
primitive ring then so is $K[G]$. Well known group
algebra results therefore easily allow us to
determine when $K[S_{n}]$ is primitive.

\begin{theorem} \label{globalprimitive}
Let $K$ be a field. The algebra $K[S_{n}]$ is right
(respectively left) primitive if and only if $K$ is
countable.
\end{theorem}

\begin{proof} Let $G$ denote the group of
fractions of $S_{n}$. Suppose $K[S_{n}]$ is right
primitive. Then, by the comments before the theorem,
$K[G]$ is right primitive. Because of
Theorem~\ref{cancel}, $G=F\times C$ with $C$ an
infinite cyclic group and $F$ a free group of rank
$n-1\geq 2$. By \cite[Theorem~ 9.1.6]{passman},
since the finite conjugacy center $\Delta (G)$ is
nontrivial, it follows that $K$ is countable.

Conversely, assume $K$ is countable. A well known
result of Formanek \cite{formanek} then states that
$R[A*B]$ is a primitive ring, if $R$ is a
$K$-algebra without zero divisors and $A*B$ is the
free product of two nonidentity groups $A$ and $B$
not both of order $2$ and such that $|R|\leq |A*B|$.
This result is Theorem 9.2.10 in \cite[page
373]{passman}. Its proof relies on fixing a
bijection $r: A \longrightarrow R[A*B]$ (with
$r(1)=0$) and showing that there exists a suitable
map $s:A\longrightarrow R[A*B]$ such that $s(a)$
depends on $r(a)$ and it exhibits the element $a$ in
its structure. The same proof works for the algebra
$K[S_{n}]$. One writes $G=C\times F = C\times (A*B)$
with $A$ an infinite cyclic group with generator $x$
and $B$ a free group of rank $n-2$. Then, since by
assumption $K$ is countable, fix a bijective map
$r:\langle x \rangle \longrightarrow K[S_{n}]$ with
$r(1)=0$. Because this map can be extended to a
bijection $A\longrightarrow K[G]=R[A*B]$ with
$R=K[C]$, all properties of the linear map
constructed in the proof of Theorem~9.2.10 in
\cite{passman}  hold (note that $s(a)\in K[S_{n}]$
for $a\in \langle x \rangle$). Hence, the proof also
yields that $J=\sum_{1\neq a\in \langle x \rangle}
s(a)K[S_{n}]$ is a proper right ideal of $K[S_{n}]$
and every maximal right ideal $M$ of $K[S_{n}]$ such
that $J\subseteq M$ does not contain nonzero
two-sided ideals of $K[S_{n}]$. Therefore,
$K[S_{n}]$ is right primitive. Since $K[S_{n}]$ is
isomorphic with its opposite ring, we get that
$K[S_{n}]$ is primitive.
\end{proof}

Since $K[G]$ is a central localization of
$K[S_{n}]$, we know that there are two types of
prime ideals of $K[S_{n}]$: those coming from $K[G]$
(that is, the primes of  the form $Q\cap K[S_{n}]$
with $Q$  a prime ideal  of $K[G]$) and those
intersecting $S_{n}$ nontrivially. As in the case of
other classes of semigroup algebras, one may expect
that the latter play an important role in the study
of the properties of the algebra, see
\cite{bookspringer}. The first step is to look at
the minimal prime ideals of $S_{n}$ and their impact
on the structure of $K[S_{n}]$.  We will see that
$P=a_1a_2\cdots a_nS_n$ is the only such ideal and
that $K[S_n]/K[P]$ is a (left and right) primitive
ring.

\begin{lemma}\label{prime}
$P=a_{1}a_{2}\cdots a_{n}S_{n}$ is a prime ideal of
$S_n$. Moreover, every prime ideal $Q$ of $K[S_{n}]$
such that $Q\cap S_{n}\neq \emptyset $ contains $P$.
Furthermore, $K[P]$ is a height one prime ideal of
$K[S_{n}]$.
\end{lemma}

\begin{proof}
Let $a,b\in S_n\setminus P$. By the comment
preceding Theorem~\ref{cancel},  there exists at
most one generator $a_{j}$ such that $aa_{j}\in P$.
Similarly, there exists at most one generator
$a_{k}$ such that $a_{k}b\in P$. Since $n\geq 3$, it
follows that there exists $a_i\in \{ a_1,a_2,\dots
,a_n\}$ such that $aa_i^2b\in S_n\setminus P$ (again
by the comment preceding Theorem~\ref{cancel}).
Hence, $P$ is a prime ideal of $S_{n}$.

Let $Q$ be a prime ideal of $K[S_{n}]$ such that
$Q\cap S_{n}\neq \emptyset $. It easily follows that
$(a_{1}\cdots a_{n})^{k}\in Q$ for some $k\geq 1$.
As $Q$ is a prime ideal and $a_{1}\cdots a_{n}$ is
central, we get $P\subseteq Q$.

The comment before Theorem~\ref{cancel} also shows
that $K[S_n]/K[P]$ is isomorphic to a monomial
algebra. Therefore, \cite[Proposition~24.2]{book}
implies that $K[P]$ is a prime ideal of $K[S_{n}]$.

To prove that $K[P]$ is a height one prime ideal
suppose that $P_{0}$ is a nonzero prime ideal of
$K[S_{n}]$ such that $P_{0}\subseteq K[P]$. Then
$P_{0}=a_{1}\cdots a_{n}P_{1}$ for an ideal $P_{1}$
of $K[S_{n}]$. Since $a_{1}\cdots a_{n}$ is central
and $P_{1}\not \subseteq P_{0}$ (by an easy degree
argument), it follows that $a_{1}\cdots a_{n}\in
P_{0}$. Hence $P_{0}=K[P]$, as desired.
\end{proof}

Note that if $Q$ is a prime ideal of $S_{n}$, then
$K[Q]$ is a prime ideal of $K[S_{n}]$. Indeed, by
the proof of Lemma~\ref{prime}, $P\subseteq Q$ and,
hence, by \cite[Proposition~24.2]{book},
$K[S_n]/K[Q]$ is a prime monomial algebra.

\begin{proposition}
Assume that $Q$ is a finitely generated prime ideal
of $S_n$. Then $K[S_n]/K[Q]$ is either a prime
PI-algebra or a (left and right) primitive ring. In
particular, $K[S_n]/K[P]$ is a (left and right)
primitive ring.
\end{proposition}

\begin{proof}
 As mentioned above, $K[S_n]/K[Q]$ is a prime finitely
presented monomial algebra. Hence, by
\cite[Theorem~1.2]{bell}, $K[S_n]/K[P]$ is either
primitive or PI.

 Since $n\geq 3$ and
$S_n/P$ contains a free monoid of rank $n-1$, it is
clear that $K[S_n]/K[P]$ does not satisfy any
polynomial identity. Therefore $K[S_n]/K[P]$ is
primitive.
\end{proof}

\section{General case and open problems}

In this section we propose some general problems
concerning the group $G_{n}(H)$ and the algebra
$K[S_n(H)]$, for an arbitrary subgroup $H$ of
$\Sym_{n}$ with $n\geq 3$. To simplify notation,
throughout this section  we put
 \begin{eqnarray}
 M&=&\langle a_1,a_2,\dots ,
a_n \mid a_1a_2\cdots
a_n=a_{\sigma(1)}a_{\sigma(2)}\cdots a_{\sigma(n)},
\;  \sigma\in H\rangle . \label{monoidpresentation}
\end{eqnarray}
In order to motivate some of the problems, we first
describe the structure of the semigroup algebra
$K[S_{n}(\Sym_{n})]$. This is another extreme case,
beyond the case studied in the previous section.

Let $\rho $ be the least cancellative congruence on
$M$. As in \cite{book}, we put $$I(\rho)=\lin_K\{
s-t\mid s,t\in M \mbox{ and } s\rho t\} ,$$ the
kernel of the natural epimorphism
$K[M]\longrightarrow K[M/\rho ]$. Note that, if
$z=a_{1}\cdots a_{n}$ is a central element of $M$,
then $\rho $ can be described by the condition:
$s\rho t$ if and only if there exists a nonnegative
integer $i$ such that $sz^{i}=tz^{i}$.

\begin{proposition} \label{symmetriccase}
Let $M=S_{n}(\Sym_{n})$. The algebra $K[M]$ is the
subdirect product of the commutative polynomial
algebra $K[a_{1},\ldots ,a_{n}]\cong K[M/\rho]$ and
the primitive monomial algebra $K[M]/K[Mz]$.
\end{proposition}
\begin{proof} Because $M$ is a homomorphic image of $S_{n}(\gr (\{ (1\, ,2\, , \ldots
,n)\} ))$, from Section~\ref{example} we know that
$z=a_{1}\cdots a_{n}\in K[M]$ is central. It easily
is verified that $zM$ is a prime ideal of $M$ and
thus, by \cite[Proposition~24.2]{book}, $K[M]/K[zM]$
is a prime ring. Clearly, the commutator ideal $P$
of $K[M]$ is a prime ideal with
$K[M]/P=K[a_{1},\ldots , a_{n}]$ and
$z(a_{i}a_{j}-a_{j}a_{i})=0$ for all $i\neq j$.
Hence $K[Mz]P=0$ and $P\subseteq I(\rho )$. But then
for $(zs,zt)\in \rho$, with $s,t\in M$, we get that
$s-t\in P$. Hence $zs-zt=z(s-t)=0$ and thus
$zs-zt=0$. It follows that $K[Mz]\cap P\subseteq
K[Mz]\cap I(\rho )= 0$. So, $K[M]$ is semiprime and
it  is the subdirect product of $K[a_{1},\ldots ,
a_{n}]$ and the prime ring $K[M]/K[Mz]$. From
Section~\ref{example} we know that the latter is
isomorphic to a monomial algebra. Since $n\geq 3$
and $M/Mz$ contains a free monoid of rank $n-1$, it
is clear that $K[M]/K[Mz]$ does not satisfy a
polynomial identity. Hence, by
\cite[Theorem~1.2]{bell}, $K[M]/K[Mz]$ is primitive.
\end{proof}

It follows that the study of $K[S_{n}(\Sym_{n})]$ is
in some sense reduced to a monomial algebra and an
algebra of a cancellative semigroup, which in this
case is free abelian. At the group level, the
intermediate cases between the case studied in
Section~\ref{example} and the case of
Proposition~\ref{symmetriccase} lead to a certain
hierarchy. Namely, let $H_{0}$ be the cyclic
subgroup of $\Sym_{n}$ generated by the cycle $(1,\;
2, \ldots , \; n)$ and let $H$ be a subgroup of
$\Sym_{n}$ that contains $H_{0}$. We have natural
monoid epimorphisms $$S_n(H_{0})\longrightarrow
S_n(H)\longrightarrow S_n(\Sym_n)$$ and group
epimorphisms
 $$G_n(H_{0})\longrightarrow
G_n(H)\longrightarrow G_n(\Sym_n).$$ We know from
Theorem~\ref{cancel} that $G_n(H_{0})=F_{n-1}\times
\mathbb{Z}$, with $F_{n-1}=\gr (a_{1},\ldots ,
a_{n-1})$ a free group of rank $n-1$, and clearly we
also have that $G_n(\Sym_n)\cong\mathbb{Z}^n$. We
claim that $$G_n(H)=F\times \mathbb{Z},$$ for a
group $F$ such that the above maps yield maps
$$F_{n-1}\longrightarrow F \longrightarrow
\mathbb{Z}^{n-1},$$ where $F$ is a group defined by
a presentation of type (\ref{present}) on $F_{n-1}$.
Indeed, consider any of the defining relations
$a_1a_2\cdots a_n=a_{\tau(1)}a_{\tau(2)}\cdots
a_{\tau(n)}$ (with $\tau \in H$)  for $S_n(H)$. Let
$k$ be such that $\tau (k)=1$ and let
$z=a_1a_2\cdots a_n$. Then, in $G_n(H)$, we have
$$\overline{a}_{1}=\overline{z}\,\overline{a}_{n}^{-1}\cdots
\overline{a}_{2}^{-1},\,
\overline{z}=\overline{a}_{\tau(1)}\cdots
\overline{a}_{\tau(k-1)}\overline{z}\,
\overline{a}_{n}^{-1}\cdots
\overline{a}_{2}^{-1}\overline{a}_{\tau(k+1)}\cdots
\overline{a}_{\tau(n)}$$ and, since $\overline{z}$
is central, the original relation is equivalent to
the relation
\begin{eqnarray} \label{second}
\overline{a}_{2}\cdots \overline{a}_{n}
=\overline{a}_{\tau(k+1)}\cdots
\overline{a}_{\tau(n)}\overline{a}_{\tau(1)}\cdots
\overline{a}_{\tau(k-1)}. \end{eqnarray} Hence $F$
is defined by all such relations with $\tau$ running
over $H$. Note that, if $\sigma=(1,2,\dots ,n)$ then
$\tau \sigma^{i}\in H_{1}=\{ \chi \in H \mid \chi
(1)=1\}$ for some $i$ and the relation $a_1a_2\cdots
a_n=a_{\tau
\sigma^{i}(1)}a_{\tau\sigma^{i}(2)}\cdots
a_{\tau\sigma^{i}}(n)$ also leads to (\ref{second}).
Therefore $F$ is the group defined with $n-1$
generators and a system of relations as in
(\ref{present}) with $H_{1}$ identified with a
subgroup of $\Sym_{n-1}$. However, we do no longer
have the assumption that $H_{1}$ contains a cycle of
length $n-1$.

For an arbitrary subgroup $H$ of $\Sym_{n}$, it
remains a challenging problem to determine the
structure of $S_{n}(H)$, $G_{n}(H)$ and of the
algebra $K[S_{n}(H)]$. We propose some concrete
problems to be investigated. For a ring $R$, we
denote by $\mathcal{J}(R)$ its Jacobson radical and
by $\mathcal{B}(R)$ its prime radical.

\begin{enumerate}
\item Describe the structure of $G_{n}(H)$. Is it a residually finite group?

\item  When does every element of $S_{n}(H)$ have a unique canonical
form, as is the case of the monoid $S_{n}(\gr (\{
(1\, ,2\, , \ldots ,n)\} ))$ considered in
Section~\ref{example}?

\item When is $S_{n}(H)$ cancellative and when does it satisfy the Ore
condition?

\item  Does there exist a congruence $\eta$ on $S_{n}(H)$ so that
 $\mathcal{B}(K[S_{n}(H)])=I(\eta)$?

\item Do we have $\mathcal{J}(K[S_{n}(H)])=\mathcal{B}(K[S_{n}(H)])$? Is this ideal
nilpotent? Is it finitely generated?

\item  Describe the minimal prime spectrum of the algebra
$K[M]$.

Problems 4,5 and 6 are motivated by the nice
behavior of the prime ideals and of the radical in
some other interesting classes of finitely presented
algebras defined by homogeneous relations,
\cite{gat-jes-okn,jaszu-okn,jes-okn-Itype}. The
results of this paper might indicate that a possible
approach can be based on the following two
constructions, that are of independent interest: the
algebras $K[M/\rho ]$, where $\rho $ is the least
cancellative congruence on $M$ and the algebras
$K[M]/K[I]$ for an ideal $I$ of $M$, the latter
potentially related to a monomial algebra.
\end{enumerate}

 \vspace{30pt}
 \noindent \begin{tabular}{llllllll}
 F. Ced\'o && E. Jespers  \\
 Departament de Matem\`atiques &&  Department of Mathematics \\
 Universitat Aut\`onoma de Barcelona &&  Vrije Universiteit Brussel  \\
08193 Bellaterra (Barcelona), Spain    && Pleinlaan
2, 1050 Brussel, Belgium \\
 cedo@mat.uab.es && efjesper@vub.ac.be\\
   &&   \\
J. Okni\'nski &&  \\ Institute of Mathematics &&
\\ Warsaw University&& \\ Banacha 2&& \\ 02-097
Warsaw, Poland &&\\ okninski@mimuw.edu.pl&&
\end{tabular}
\end{document}